\documentclass{amsart}

\usepackage{times}
\usepackage[T1]{fontenc}
\usepackage{tikz-cd} 
\usepackage[a4paper,top=2cm,bottom=2cm,left=2.5cm,right=2.5cm,marginparwidth=1.75cm]{geometry}

\usepackage{amsmath,amssymb,amsfonts,amsthm,upgreek}
\usepackage{mathrsfs, mathtools}
\usepackage{graphicx}
\usepackage[colorinlistoftodos,prependcaption,textsize=footnotesize,textwidth=2.1cm]{todonotes}
\usepackage[all]{xy}
\usepackage{cancel}
\usepackage{appendix}
\usepackage{enumerate}
\usepackage{multicol}
\usepackage[colorlinks=true, allcolors=blue,backref=page]{hyperref}
\usepackage[initials,alphabetic]{amsrefs}

\newcommand{\rG}{{\rm G}}

\newcommand{\cH}{\mathcal{H}}

\newcommand{\cL}{\mathcal{L}}

\newcommand{\sX}{\mathscr{X}}

\newcommand{\R}{\mathbb{R}}

\newcommand{\SO}{{\rm SO}}

\newcommand{\SU}{{\rm SU}}

\newcommand{\U}{{\rm U}}

\newcommand{\End}{{\mathrm{End}}}

\renewcommand{\epsilon}{\varepsilon}

\renewcommand{\Im}{\mathop{\mathrm{Im}}}

\renewcommand{\Re}{\mathop{\mathrm{Re}}}

\newcommand{\vol}{\mathrm{vol}}

\newcommand{\qandq}{\quad\text{and}\quad}
\newcommand{\qwithq}{\quad\text{with}\quad}
\newcommand{\qforq}{\quad\text{for}\quad}
\newcommand{\qonq}{\quad\text{on}\quad}
\newcommand{\qasq}{\quad\text{as}\quad}

\def\<{\mathopen{}\left<}
\def\>{\right>\mathclose{}}
\def\({\mathopen{}\left(}
\def\){\right)\mathclose{}}

\usepackage{multicol, color}

\definecolor{gold}{rgb}{0.85,.66,0}
\definecolor{cherry}{rgb}{0.9,.1,.2}
\definecolor{burgundy}{rgb}{0.8,.2,.2}
\definecolor{orangered}{rgb}{0.85,.3,0}
\definecolor{orange}{rgb}{0.85,.4,0}
\definecolor{olive}{rgb}{.45,.4,0}
\definecolor{lime}{rgb}{.6,.9,0}
\definecolor{green}{rgb}{.2,.7,0}
\definecolor{grey}{rgb}{.4,.4,.2}
\definecolor{brown}{rgb}{.4,.3,.1}

\theoremstyle{plain}
\newtheorem{theorem}{Theorem}
\newtheorem{proposition}{Proposition}
\newtheorem{corollary}[proposition]{Corollary}
\newtheorem{lemma}[proposition]{Lemma}

\newtheorem{remark}[proposition]{Remark}
\newtheorem{definition}[proposition]{Definition}
\newtheorem{example}[proposition]{Example}

\newtheorem*{remark*}{Remark}

\title{Isolated singularities in $\rG_2$-structures with torsion}
\author{Henrique Sá Earp, Jakob R. Stein}
\date{\today}

\begin{document}

\begin{abstract}
We revisit the study of $\rG_2$-structures with special torsion, and isolated singularities. Many of the known examples with conical singularities admit additional symmetries, and we describe circle-invariant $\rG_2$-structures in this context. Finally, we show that collapsing the circle fibres of a contact Calabi-Yau manifold at isolated points cannot produce a $\rG_2$-structure with bounded torsion.
\end{abstract}
\maketitle

\tableofcontents



\section{Introduction}
The special role of conical singularities in the study of exceptional holonomy groups goes back to the first constructions of $\rG_2$-holonomy Riemannian metrics by Bryant \cite{Bryant1987}. Their one-parameter families of complete desingularisations in \cite{Bryant1989} were later interpreted by Karigiannis et al. \cites{Karigiannis2009,Karigiannis2020}, based on earlier insights by Joyce \cite{Joyce1996}, as a local model describing the codimension-one stratum of the boundary to the moduli-space of torsion-free $\rG_2$-structures on some fixed smooth compact $7$-manifold.  
 
The resulting singular degenerations are referred to as $\rG_2$-\emph{conifolds}: despite the thorough analytic description of resolutions of their singularities in  \cite{Karigiannis2009} and their expected abundance, as of yet, no such compact manifolds with holonomy $\rG_2$, and \emph{isolated conical singularities} (ICS) have been shown to exist, see Definition \ref{def:CSG2}.

Aside from the $\rG_2$-cones themselves, the only known simply-connected example, that is otherwise complete, is the non-compact cohomogeneity-one solution by Foscolo-Haskins-Nordstr\"{o}m \cite{FoscoloALC}. 
Of course, one can also obtain more examples by taking finite quotients \cite{Cortes2015}, but it is worth remarking that, unlike Riemannian holonomy groups in even dimensions, $\rG_2$-structures arising from orbifolds never have isolated conical singularities. This fact follows from a corollary to the classification of space-forms \cite{Wolf2011}: there are no discrete subgroups of $\SO(2n+1)$ with isolated fixed points in $\R^{2n+1}$.

Foundational works on Riemannian manifolds with ICS include \cites{Cheeger1983, Adams1988, Lockhart1985}, which studies their de Rham cohomology and related elliptic equations. Outside of the $\rG_2$ setting, there are a number of more recent results studying special Riemannian manifolds with ICS, e.g. Einstein, or K\"{a}hler Ricci-flat. The first examples of compact Ricci-flat metrics with non-orbifold ICS, were constructed by Sun--Hein in \cite{Hein2017} using complex geometry. The work of Ozuch on Einstein 4-manifolds \cite{Ozuch2022} raises the possibility that not every Einstein conifold metric can even arise as a degenerate limit.   Moreover, various definitions of a conical singularity exist in the literature, depending on the context. We will follow \cite{Karigiannis2009}*{Definition 2.31}, see also \cite{Pacini2013} for a general introduction to desingularizing geometries with ICS. 

Aside from the mathematical motivation to study special conifold metrics, one should mention the interest in such spaces coming from theoretical physics, e.g. $\rG_2$-holonomy conifolds can give rise to chiral fermions in suitable regimes of M-theory \cite{Acharya2004}. 

Let us establish from the outset what we mean by $\rG_2$-structure with ICS. Let $N$ be a smooth, compact 6-manifold, equipped with an $\SU(3)$-structure $(\omega, \Upsilon)$, see \S \ref{sec: SU(3)structures}. 
The \emph{conical} $\rG_2$-\emph{structure} $\varphi_C, \psi_C$ on $\mathbb{R}_{>0} \times N$ is given at $\{r\}\times N$ by: 
\begin{equation}
  \varphi_C= r^2 dr \wedge \omega + r^3 \Re{\Upsilon}
  \qandq
  \psi_C = \tfrac{1}{2} r^4 \omega^2 - dr \wedge  \Im \Upsilon.
\end{equation}
Denote by $g$ the Riemannian metric on $N$ induced by the $\SU(3)$-structure, and by $\vol$ the associated volume form. Notice that the $\rG_2$-structure $3$-form  $\varphi_C$ is homogenous under rescaling, i.e. $\mathcal{L}_{r \partial r} \varphi_C = 3 \varphi_C$, and it induces a Riemannian cone metric $g_C$ on $\mathbb{R}_{>0} \times N$, with a conical volume form $\vol_C$ given respectively by 
\[
g_C = dr^2 + r^2 g
\qandq \vol_C = r^6 dr \wedge \vol.    
\]
We will denote this Riemannian manifold by $C(N):=(\mathbb{R}_{>0} \times N,g_C)$, with $i: N \hookrightarrow C(N)$ being the inclusion at the level $r=1$. We will refer to $N$ as the \emph{link} of this cone. 

\begin{definition} \label{def:CSG2} Suppose $M^7$ is a smooth manifold away from a  (non-empty) discrete set of points $S \subset M$. A smooth $\rG_2$-structure $\varphi\in\Omega^3_+(M\setminus S)$ has an \emph{isolated conical singularity} (ICS) at $p \in S$, with rate $\nu>0$, if there exists a quadruple $(U,\varphi_C, \epsilon, f)$ given by
\begin{itemize}
    \item $U \subset M \setminus S$ is smooth punctured open neighbourhood of $p$,
    \item $\varphi_C$ is a conical $\rG_2$-structure on $C(N)$ for some $N^6$, 
    \item $\epsilon>0$ is a real constant, 
    \item $f:(0,\epsilon) \times N \rightarrow U$ is a diffeomorphism,
\end{itemize}
    such that, with respect to the induced cone metric $g_C$, 
 \[
    | \nabla^j_C (f^* \varphi - \varphi_C)|_{g_C} = O (r^{\nu-j}),
    \quad\forall j\geq0. 
 \]
\end{definition}
\begin{remark}
    A (smooth) $\rG_2$-structure is called \emph{torsion-free} if $\varphi$ is both closed and co-closed. In this case, then there is an additional gauge-fixing condition on the diffeomorphism $f$, for which it suffices to control the decay of derivatives up to $j=1$, by elliptic regularity, cf. \cite{Karigiannis2009}.
\end{remark}
\begin{remark}
    Although this definition has the advantage of a good deformation theory, the construction by Chen \cite{GaoChen2018} on the unit ball in the cone over the flag manifold $\SU(3)/T^2$, shows that this definition may not be the most general behaviour one could expect: Chen's example has non-polynomial decay rate to the model structure.  
\end{remark}

In this text, we will attempt to highlight the difficulty in producing compact $\rG_2$-manifolds with isolated singularities, by studying the easier problem of producing conically singular $\rG_2$-structures with torsion. We will revisit some familiar constructions of circle-invariant $\rG_2$-structures in this context, e.g. the Apostolov-Salamon equations \cite{Apostolov2004}, describing circle-invariant $\rG_2$-structures in the torsion-free case, cf. \cite{Acharya2020}. 

Known examples of conically singular $\rG_2$-structures with special torsion are furnished e.g. by the sine-cone construction, cf. \cites{Bilal2003, FoscolonK}: given a nearly-K\"{a}hler 6-manifold $(N, g_\mathrm{nk})$, then $(0,\pi) \times N$ is naturally equipped with a co-closed $\rG_2$-structure $\varphi$ inducing the Einstein metric $dt^2 + \sin (t)^2 g_\mathrm{nk}$. In particular, $\varphi$ is a \emph{nearly}-$\rG_2$-structure, i.e. $d \varphi = \tau_1 * \varphi$ for some constant $\tau_1$, and moreover has isolated singularities at $0$ and $\pi$ modelled on the torsion-free $\rG_2$-cone over $N$. Starting with $N$ nearly-K\"{a}hler produces a \emph{strictly} nearly-$\rG_2$-structure on $(0,\pi) \times N$, i.e. there is a unique (up to scale) Killing spinor for the induced metric \cite{Friedrich1995}: however, taking $N$ fibred by flat tori gives nearly-$\rG_2$ sine-cones induced by a 3-Sasaki, or Sasaki-Einstein structures. \\

\paragraph{\textbf{Outline and main results}} 

While much of the background material on the geometric structures studied in this note can be found elsewhere in the literature, see e.g. \cite{Dwivedi2025}*{\S 2}, we begin by fixing notation, giving the necessary details in \S \ref{sec:preliminaries}. 

Starting with $\rG_2$-structures on $7$-manifolds in \S \ref{sec:G2structures}, we will describe their torsion as a quadruple of forms $(\tau_1, \tau_7, \tau_{14}, \tau_{27})$, where the subscript denotes the dimension of the corresponding $\rG_2$-representation. 
As $\SU(3)$-structures arise in the circle-invariant reduction, we recall some related material on these structures in \S\ref{sec: SU(3)structures}, and describe their torsion in terms of a set of forms $(\upsilon_1, \hat{\upsilon}_1, \upsilon_6, \hat{\upsilon}_6, \upsilon_8, \hat{\upsilon}_8, \upsilon_{12})$. Finally, when there is an additional non-vanishing vector field on the 7-manifold, such as the radial vector field in the conical setting, there is a further reduction to $\SU(2)$-structures on 5-manifolds, which are explained in \S \ref{sec: SU(2)structures}.

We then focus on circle-invariant $\rG_2$-structures in \S \ref{sec:s1Ansatz}. On the complement of the fixed points of the circle action, there is a natural basic $\SU(3)$-structure $(\omega, \Upsilon)$ on the quotient space, a function $t$ giving the length of the circle orbit, and a connection 1-form $\theta$. We explicitly compute the torsion of a circle-invariant $\rG_2$ structure purely in terms of this basic data in Proposition \ref{prop:invariantg2torsion}, cf. \cite{Salamon2002}*{Theorem 3.1} 

With this general set-up, we then write the above data as a set of evolution equations on the level set $\Sigma := t^{-1}(t_\mathrm{reg})$ of a regular value $t_\mathrm{reg}>0$ of $t$, by imposing that the corresponding $S^1$-invariant $\rG_2$-structure is either closed, or co-closed. We then further specialise to the case that $(\omega, \Upsilon)$ is a fixed cone over $\Sigma$ with radial parameter $t$, giving a set of evolution equations for $\theta$ on the link. 

These computations allow us to give an example of a one-parameter family of co-closed asymptotically conical $\rG_2$-structures obtained via this construction, which has the torsion-free Bryant-Salamon cone over $S^3 \times S^3$ appearing as a limit. The following theorem appears in more detail as Proposition \ref{thm:gamma} in \S \ref{sec:s1Ansatz}: 

\begin{theorem} 
\label{thm: BS cone examples}
    Let $M^7= \R_{>0}\times S^3 \times S^3$. There is a one-parameter family $\lbrace(\varphi_\gamma, \psi_\gamma)\rbrace_{\gamma\geq0 }$ of co-closed $\rG_2$-structures on $M$ such that $(\varphi_0, \psi_0)$ is the torsion-free Bryant-Salamon cone and, with respect to the induced metric $g_\gamma$,
\begin{itemize}   
    \item Close to the singular end, as $r \rightarrow 0$, we have $|\varphi_\gamma- \varphi_0| = \Theta(1)$, $|\psi_\gamma- \psi_0| = \Theta(r^{-1})$. 
     \item Away from the singular end, as $r \rightarrow \infty$, we have $|\varphi_\gamma- \varphi_0| = \Theta(r^{-3})$, $|\psi_\gamma- \psi_0| = \Theta(r^{-4})$.
\end{itemize}
\end{theorem}
\begin{remark*}
In Theorem \ref{thm: BS cone examples}, and throughout the text, we use the notation $f =\Theta (g)$ for real-valued functions $f,g$ such that $f= O(g)$, and $g=O(f)$.     
\end{remark*}

While non-compact, these examples have the property of being asymptotic to the cone while having a non-conical singularity as $t\rightarrow 0$. This suggests that it may be interesting to weaken even the requirement that the isolated singularity be conical.

In order to explore this idea in the compact setting, we will focus on a particular class of compact $S^1$-invariant $\rG_2$-structures in \S \ref{sec:cCYAnsatz}, so-called \emph{contact Calabi-Yau} manifolds, cf. \cite{SaEarp2020}. In terms of the notation of \S \ref{sec:s1Ansatz}, here $t$ is constant, $d \theta = \omega$ and the torsion of the transverse $\SU(3)$-structure $(\omega, \Upsilon)$ vanishes, i.e. the geometry is regular Sasakian. The freedom to vary $t$ globally on $M$ gives a one-parameter family of co-closed $\rG_2$-structures, which degenerates to a Calabi-Yau orbifold transverse to the circle-fibre in the  $t\rightarrow 0$ limit. 
 With this special case in mind, one might hope that by varying the fibre-length $t$ \emph{locally}, one could produce special singular $\rG_2$-structures with fibres degenerating, at a prescribed rate, at isolated points in the base of the fibration. We show, however, that such a scheme will always produce a $\rG_2$-structure with unbounded torsion. For example, in Lemma \ref{lem:cCYisolated} we prove that if the order of vanishing of $t$ is polynomial in the distance $r$ to the base point of the vanishing fibre, then norm of torsion behaves like $\Theta(r^{-1})$ to leading order as $r \rightarrow 0$. In particular, as we show in Lemma \ref{lem:nonexistence}, this order of vanishing must hold for isolated conical singularities. We summarise the combined results of these two lemmas in the following proposition:

\begin{proposition} Contact Calabi-Yau manifolds do not admit $S^1$-invariant $\rG_2$-structures with bounded torsion, compatible with the transverse Calabi-Yau structure, and fibres collapsing at isolated points. 
 \end{proposition}

Together with Theorem \ref{thm: BS cone examples}, this conclusion somewhat frames the problem of constructing special singular $\rG_2$-structures within a compromise between bounded torsion and a strictly conical singular profile. Our study thus seems to provide a heuristic illustration of how difficult the full problem actually is.\\

\paragraph{\textbf{Acknowledgements.}}
The authors would like to thank Jason Lotay, Thibault Langlais, Viktor Majewski, and Spiro Karigiannis for valuable discussions and insights. 

HSE was supported by the  São Paulo Research Foundation (Fapesp)    \mbox{[2021/04065-6]} \emph{BRIDGES collaboration} and the Brazilian National Council for Scientific and Technological Development (CNPq)  \mbox{[311128/2020-3]}.

JS was supported by the São Paulo Research Foundation (Fapesp) \mbox{[2023/02809-3]} linked to the \mbox{[2021/04065-6]} \emph{BRIDGES collaboration}.

\section{Preliminaries on geometric structures} 
\label{sec:preliminaries}

\subsection{\texorpdfstring{$\rG_2$-structures}{G2-structures}}

\label{sec:G2structures}

Let us begin with $\rG_2$-structures. We recall the following standard results in the literature, most of which can be found in \cite{Karigiannis2003}, \cite{Karigiannis2007}, or \cite{Bryant2006}, see also references therein. Given a non-degenerate 3-form $\varphi$ on an open set in $\R^7$, we can define a Riemannian metric $g$, and a volume $\mathrm{Vol}_g$ such that: 
\begin{align}
   g(X,Y) \mathrm{Vol}_g = \tfrac{1}{6} (X \lrcorner\varphi )\wedge (Y \lrcorner \varphi )\wedge \varphi    
\end{align}
for any vector fields $X,Y$. 
These data also determine a 4-form $\psi := *\varphi$ such that $\langle \varphi\rangle^2 = 7$, or equivalently $\varphi \wedge * \varphi = 7 \mathrm{Vol}_g$. Under the action of the pointwise stabiliser $\rG_2$ of $\varphi$, we have the associated splitting of forms:
\begin{equation}
\label{eq:g2forms}
\Omega^4\left( \mathbb{R}^7 \right) = \Omega^4_1 \oplus \Omega^4_7 \oplus \Omega^4_{27}
\quad\qandq\quad
\Omega^5\left( \mathbb{R}^7 \right) = \Omega^5_7 \oplus \Omega^5_{14},
\end{equation}
where:
\[
\renewcommand{\arraystretch}{1.3}
\begin{array}{@{}l@{\;=\;}l}
\Omega^4_1     & \left\{ f * \varphi \mid f \in \Omega^0 \right\}, \\
\Omega^4_7     & \left\{ X \wedge \varphi \mid X \in \Omega^1 \right\}, \\
\Omega^4_{27}  & \left\{ *\beta \mid \beta \in \Omega^3,\ \beta \wedge *\varphi = 0, \beta \wedge \varphi =0  \right\}
\end{array}
\quad\qandq\quad
\begin{array}{@{}l@{\;=\;}l}
\Omega^5_7     & \left\{ X \wedge *\varphi \mid X \in \Omega^1 \right\}, \\
\Omega^5_{14}  & \left\{ *\alpha \mid \alpha \in \Omega^2,\ *\alpha = -\alpha \wedge \varphi \right\}.
\end{array}
\]
Accordingly, the torsion $\tau=(\tau_1, \tau_7, \tau_{14}, \tau_{27})$ of a $\rG_2$-structure $\varphi$ satisfies 
\begin{equation} 
\label{eq:g2torsion}
    d\varphi = \tau_1 * \varphi + 3 \tau_7 \wedge \varphi + *\tau_{27}
    \quad\qandq\quad
    d{*}\varphi = 4 \tau_7 \wedge {*}\varphi + {*}\tau_{14}.
\end{equation}
In practice, one computes the torsion components by way of the following lemma:
\begin{lemma} 
\label{lem:g2torsion} 
    The torsion components $\tau_1, \tau_7, \tau_{14}$ are given in terms of $d \varphi$ and $d^* \varphi = - *d * \varphi$, by:
\begin{align*}
\tau_1     &= \tfrac{1}{7} * ( d \varphi \wedge \varphi), \\
\tau_7     &= \tfrac{1}{12} *(d^* \varphi \wedge * \varphi) = - \tfrac{1}{12} *((* d \varphi) \wedge \varphi), \\
\tau_{14}  &= \tfrac{1}{3} * \left( 2 d * \varphi + d^* \varphi \wedge \varphi \right).
\end{align*}
\end{lemma}
Moreover, under conformal changes, we have

\begin{lemma}\cite{Karigiannis2003}*{Theorem 3.1.4}
    Under the conformal change $\tilde{\varphi} = f^3 \varphi$, for some function $f>0$, we have the associated structures $\tilde{g} = f^2 g$, $\tilde{\psi} = f^4 \psi$, and the new torsion components $\tilde{\tau}_i$ of $\tilde{\varphi}$ are:
    \begin{align*}
        \tilde{\tau}_1     &= f^{-1} \tau_1,       & \tilde{\tau}_7     &= \tau_7 + d (\ln{f}), \\
        \tilde{\tau}_{27}  &= f^2 \tau_{27},       & \tilde{\tau}_{14}  &= f \tau_{14}.
    \end{align*}
    Moreover, the norms of these torsion components with respect to the new metric $\tilde{g}$ are:
    \begin{align*}
        \left| \tilde{\tau}_{14} \right|_{\tilde{g}} = f^{-1} \left| \tau_{14} \right|_g, \quad
        \left| \tilde{\tau}_{7} \right|_{\tilde{g}}  = f^{-1} \left| \tau_{7} + d (\ln{f}) \right|_g, \quad
        \left| \tilde{\tau}_{27} \right|_{\tilde{g}} = f^{-1} \left| \tau_{27} \right|_g.
    \end{align*}
\end{lemma}
\subsection{\texorpdfstring{$\SU(3)$-structures}{SU(3)-structures}}

\label{sec: SU(3)structures}

Following \cite{FoscoloALC}, recall that an $\SU(3)$-structure on a six-manifold is defined by a pair $(\omega, \Upsilon)$ of a real positive 2-form $\omega$ and a complex 3-form $\Upsilon = \Re\Upsilon + i \Im  \Upsilon$, such that 
\begin{align}
\omega \wedge \Upsilon=0 \quad\qandq\quad \tfrac{1}{6} \omega^3 = \tfrac{1}{4} \Re\Upsilon \wedge \Im  \Upsilon = \vol_g    
\end{align}
where $\vol_g$ is the volume form of the metric $g$ induced by the $\SU(3)$-structure. With respect to the induced almost complex structure $J$, this metric satisfies $g(X,Y) = \omega (X, JY)$ on tangent vectors $X,Y$. 

We have the associated splitting of forms: 
\begin{align}
\label{eq:CYforms}
\Omega^2\left( \mathbb{R}^6 \right) = \Omega^2_1 \oplus \Omega^2_6 \oplus \Omega^2_8
\quad\qandq\quad
\Omega^3\left( \mathbb{R}^6 \right) = \Omega^3_6 \oplus \Omega^3_{1 \oplus 1} \oplus \Omega^3_{12}
\end{align}
where
\[
\renewcommand{\arraystretch}{1.3}
\begin{array}{@{}r@{\;=\;}l}
\Omega^2_1  & \left\{ f \omega \mid f \in \Omega^0 \right\}, \\
\Omega^2_6  & \left\{ X \lrcorner\, \Re \Upsilon \mid X \in \Omega^1 \right\}, \\
\Omega^2_8  & \left\{ \alpha \in \Omega^2 \mid \alpha \wedge \omega^2 = \alpha \wedge \Upsilon = 0 \right\}
\end{array}
\quad\qandq\quad
\renewcommand{\arraystretch}{1.3}
\begin{array}{@{}c@{\;=\;}l}
\Omega^3_6              & \left\{ X \wedge \omega \mid X \in \Omega^1 \right\}, \\
\Omega^3_{1 \oplus 1}   & \left\{ f_1 \Re \Upsilon + f_2 \Im \Upsilon \mid f_1, f_2 \in \Omega^0 \right\}, \\
\Omega^3_{12}           & \left\{ \alpha \in \Omega^3 \mid \alpha \wedge \omega = \alpha \wedge \Upsilon = 0 \right\}
\end{array}
\]
Moreover, given the compatible metric $g$, we have the following identities for the Hodge star operator $*$, where we still denote respectively by $X$ and $JX$ their image 1-forms $X^\flat$ and  $(JX)^\flat$ under the flat isomorphism:  
\[
\begin{aligned}
  *\omega &= \tfrac{1}{2} \omega^2, \qquad
  *\Re \Upsilon = \Im \Upsilon, \qquad
  *(X \wedge \Re \Upsilon) = - X \lrcorner \Im \Upsilon, \qquad
  *(X \wedge \omega^2) = 2 J X.
\end{aligned}
\]
We also have the following identities for the almost complex structure $J$:
\[
\begin{aligned}
  X \lrcorner \omega = J X, \qquad
  J X \lrcorner \Im \Upsilon = X \lrcorner \Re \Upsilon = * ( X \wedge \Im \Upsilon ), \qquad
  * (X \wedge \omega) = J X \wedge \omega,
\end{aligned}
\]
and the following norms:
\begin{align}
\label{eq:CYnorms}
  | \omega |^2 = 3, \qquad
  | \Re{\Upsilon} |^2 = | \Im{\Upsilon} |^2 = 4, \qquad
  | X \wedge \omega |^2 = 2 |X|^2, \qquad
  | X \lrcorner \Re{\Upsilon} |^2 = 2 |X|^2.
\end{align}
As immediate corollaries of the previous identities, one obtains:
\begin{lemma} 
    Let $(\omega, \Upsilon)$ be an $\SU(3)$-structure, $X$ a vector field. Then:
\begin{enumerate}[(i)]
    \item $(X \lrcorner  \Re{\Upsilon}) \wedge \omega = -JX \wedge \Re{\Upsilon}$.
    \item $(X \lrcorner \Re{\Upsilon}) \wedge \Re{\Upsilon} = X \wedge \omega^2$.
    \item $(X \lrcorner \Re{\Upsilon}) \wedge \Im{\Upsilon} = JX \wedge \omega^2$.
\end{enumerate}
\end{lemma}
The intrinsic torsion of an $\SU(3)$-structure can be expressed by the multiplet $(\upsilon_1, \hat{\upsilon}_1, \upsilon_6, \hat{\upsilon}_6, \upsilon_8, \hat{\upsilon}_8, \upsilon_{12})$, where 
\[
\upsilon_1, \hat{\upsilon}_1 \in \Omega^0_1, \quad
\upsilon_6, \hat{\upsilon}_6 \in \Omega^1_6, \quad
\upsilon_8, \hat{\upsilon}_8 \in \Omega^2_8, \quad
\upsilon_{12} \in \Omega^3_{12},
\]
and indeed:
\begin{gather} 
\label{eq:CYtorsion}
\begin{aligned}
    d \omega &= 3 \upsilon_1 \Re \Upsilon + 3 \hat{\upsilon}_1 \Im  \Upsilon + \upsilon_{12} + \upsilon_{6} \wedge \omega, \\ 
    d \Re \Upsilon &= 2 \hat{\upsilon}_1 \omega^2 + \hat{\upsilon}_6 \wedge \Re \Upsilon +  \upsilon_8 \wedge \omega,\\
    d \Im  \Upsilon &= - 2 \upsilon_1 \omega^2 + \hat{\upsilon}_6 \wedge \Im  \Upsilon +  \hat{\upsilon}_8 \wedge \omega.\\
\end{aligned}   
\end{gather}
For the purposes of this text, a  \emph{Calabi-Yau} structure is a torsion-free $\SU(3)$-structure, i.e. such that $d \omega = 0 $, $d \Im  \Upsilon  = 0$, and  $d \Re \Upsilon =0$. 

\subsection{\texorpdfstring{$\SU(2)$-structures}{SU(2)-structures}}
\label{sec: SU(2)structures}

We adopt here the perspective of \cite{conti2007}*{Proposition 1} cf. \cite{Stein2023}, by characterising an $\SU(2)$-reduction of the frame bundle of a 5-manifold $P$ as the following multi-tensor structure:
\begin{definition}
    An \emph{$\SU(2)$-structure} on a 5-manifold $P$ is a triple of 2-forms $\left(\omega_1 , \omega_2 , \omega_3 \right)$ and a nowhere-vanishing 1-form $\eta$, satisfying: 
\begin{enumerate}[(i)]
    \item $\omega_i \wedge \omega_j = 2 \delta_{ij} v$, with $v\in\Omega^4(P)$ fixed such that $v \wedge \eta$ is nowhere-vanishing: i.e. $v$ is a volume form on the distribution $\mathcal{H} := \ker \eta$.
    
    \item $X \lrcorner \omega_1 = Y \lrcorner \omega_2 \Rightarrow \omega_3 (X, Y) \geq 0$, i.e. the orthonormal basis $\left( \omega_1 , \omega_2 , \omega_3 \right)$ of self-dual 2-forms on $\mathcal{H}$ is oriented, with respect to the volume of the Riemannian metric $g_\cH$ on $\mathcal{H}$ defined via
    \[
    g_\cH (X,Y) v = \omega_1 \wedge \left( X \lrcorner  \omega_2 \right) \wedge \left( Y \lrcorner \omega_3 \right),
    \qforq
    X,Y \in \cH.
    \]
\end{enumerate}
\end{definition}
An $\SU(2)$-structure $\left(\eta, \omega_1 , \omega_2 , \omega_3 \right)$ naturally defines a Riemannian metric $g = g_\cH + \eta^2$ on $P$, according to the splitting $\sX(P) = \cH \oplus \langle \xi \rangle$, where $\xi$ is the vector field 
on $P$ such that $\eta(\xi) = 1$, with $\xi \lrcorner \omega_i = 0$, $i=1,2,3$. We refer to this as the \emph{Reeb vector field}, 
with respect to $g$. Associated to the metric, there is an orthogonal splitting: 
\begin{align}
\Omega^1(\mathbb{R}^5) = \Omega^1_1 \oplus \Omega^1_4 
\quad\qandq\quad
\Omega^2(\mathbb{R}^5) = \Omega^2_{1 \oplus 1 \oplus 1} \oplus \Omega^2_3 \oplus \Omega^2_4
\end{align}
where
\[
\renewcommand{\arraystretch}{1.3}
\begin{array}{@{}r@{\;=\;}l}
\Omega^1_1 & \left\langle \eta \right\rangle, \\
\Omega^1_4 & \cH^*,
\end{array}
\quad\qandq\quad
\renewcommand{\arraystretch}{1.3}
\begin{array}{@{}r@{\;=\;}l}
\Omega^2_{1 \oplus 1 \oplus 1} & \left\{ \sum_{i=1}^3 f_i \omega_i \mid f_i \in \Omega^0 \right\}, \\
\Omega^2_3                    & \left\{ \alpha \in \Omega^2 \mid *\alpha = -\eta \wedge \alpha \right\}, \\
\Omega^2_4                    & \left\{ X \wedge \eta \mid X \in \cH^* \right\}.
\end{array}
\]
There is also an associated quaternionic triple $J_1, J_2, J_3 \in \End{(TP)}$, such that 
\[
J_i J_j = - J_j J_i = J_k,
\qforq (ijk)\sim(123),
\]
defined by
\[
\omega_i (X,Y) = g_\cH ( J_i X, Y) 
\qandq
J_i \xi = 0,
\qforq i=1,2,3.
\]
\begin{lemma}[\cite{FoscoloALC}*{Lemma 4.3}]
    With respect to the Riemannian metric $g = g_\cH + \eta^2$ on $P^5$, we have:
\[
*\eta = \tfrac{1}{2} \omega_1^2, 
\quad
* \omega_i = \eta \wedge \omega_i, 
\qandq
* X = - J_i X \wedge \eta \wedge \omega_i,
\qforq X \in \cH.
\]
\end{lemma}
So, given some 6-manifold $N$ with an $\SU(3)$-structure $\left(\omega, \Upsilon\right)$, and an oriented embedding $\iota:P \hookrightarrow M$, we obtain an $\SU(2)$-structure on $P$ from the following reductions:
\begin{equation} 
\label{eq: su2define}
    \eta = \iota^* ( \hat{n} \lrcorner \omega ), \quad
    \omega_1 = \iota^* \omega, \quad
    \omega_2 = \iota^* ( \hat{n} \lrcorner \Re \Upsilon ), \quad
    \omega_3 = \iota^* ( \hat{n} \lrcorner \Im \Upsilon ).
\end{equation}
where $\hat{n}$ is the unit vector field normal to $\iota(P) \subset N$, determined by the chosen orientation and the Riemannian metric induced by $\left( \omega, \Upsilon \right)$. 

Conversely, we can identify a tubular neighbourhood of $P\subset N$ with $P \times I$ for some interval $I \subseteq \R$ using the exponential map. In these coordinates, the metric on $N$ appears as $dt^2 + g_t$ for some $t$-dependent metric $g_t$ on $P$, and \eqref{eq: su2define} gives rise to a family of $\SU(2)$-structures $\left( \eta, \omega_i \right)_{t \in I}$ on $P$ inducing $g_t$. 
In this neighbourhood, the $\SU(3)$-structure $\left(\omega, \Upsilon\right)$ on $N$ takes the form:  
\begin{equation} 
\label{eq: CYdefine}
    \omega = dt \wedge \eta + \omega_1 \qandq
    \Upsilon = \left( dt + i \eta \right) \wedge \left( \omega_2 + i \omega_3 \right).
\end{equation}
If $\left( \omega , \Upsilon \right)$ is Calabi--Yau, i.e. $\omega$, $\Upsilon$ are closed, then $\left( \eta, \omega_i \right)_{t \in I}$ satisfy the following structure equations on $P$:
\begin{equation} 
\label{eq: hypo}
    d \omega_1 = 0, \quad
    d( \omega_3 \wedge \eta ) = 0, \quad
    d(\omega_2 \wedge \eta ) = 0,
\end{equation}
along with the evolution equations for $t \in I$:
\begin{equation} 
\label{eq: dynamic}
    d \eta = \partial_t \omega_1, \quad
    d\omega_2 = - \partial_t (\omega_3 \wedge \eta), \quad
    d\omega_3 = \partial_t ( \omega_2 \wedge \eta ).
\end{equation}
Observe that the structure equations \eqref{eq: hypo} are preserved under the evolution \eqref{eq: dynamic}. This allows us to interpret a Calabi-Yau structure (at least locally) as a flow by \eqref{eq: dynamic} in the space of $\SU(2)$-structures, satisfying \eqref{eq: hypo} on some fixed $5$-manifold. This motivates the following definition:   
\begin{definition} \cite{conti2007}*{Definition 5}
    A \emph{hypo-structure} on a 5-manifold $P$ is an $\SU(2)$-structure satisfying \eqref{eq: hypo}. We refer to \eqref{eq: dynamic} as the \emph{hypo-evolution equations}.
\end{definition}

Leaving aside completeness of the resulting metrics for a moment, the Riemannian cone is an important class of examples for this construction of Calabi-Yau metrics:

\begin{example}[Calabi--Yau cone] 
    Let $P$ be a $5$-manifold equipped with a fixed hypo-structure $\left( \eta^{se}, \omega_i^{se} \right)$.  Then the 1-parameter family $\left( \eta, \omega_i \right)_{t \in \mathbb{R}_{>0}}$ of $\SU(2)$-structures given by
    \begin{align}
    \label{eq: sesu2}
        \eta = t \eta^{se} 
        \qandq
        \omega_i = t^2 \omega_i^{se}
    \end{align}
    satisfies \eqref{eq: hypo} and \eqref{eq: dynamic} if, and only if, $\left( \eta^{se}, \omega_i^{se} \right)$ satisfy the following structure equations on $P$:
    \begin{align}
    \label{eq: sasaki}
        d \eta^{se} = 2 \omega_1^{se}, \quad
        d \omega_2^{se} = -3 \omega_3^{se} \wedge \eta^{se}, \quad
        d \omega_3^{se} =  3 \omega_2^{se} \wedge \eta^{se}.
    \end{align}
    As in \eqref{eq: CYdefine}, this family defines the \emph{conical $\SU(3)$-structure} $\left( \omega_C, \Upsilon_C \right)$ on $\mathbb{R}_{>0} \times P$ given by
    \begin{align}
    \label{eq: CYcone}
        \omega_C = t\, dt \wedge \eta^{se} + t^2 \omega^{se}_1 \qandq
        \Upsilon_C = t^2 \left( \omega^{se}_2 + i \omega^{se}_3 \right) \wedge \left( dt + i t \eta^{se} \right).
    \end{align}
    This $\SU(3)$-structure is Calabi--Yau if, and only if, $\left( \eta^{se}, \omega_i^{se} \right)$ satisfies the structure equations \eqref{eq: sasaki}.

    We refer to an $\SU(2)$-structure $\left( \eta^{se}, \omega_i^{se} \right)$ satisfying \eqref{eq: sasaki} as being \emph{Sasaki--Einstein}. One can show that such a structure induces a Sasaki--Einstein metric $g^{se}$ on $P$, or in other words, the Calabi--Yau metric $g_C$ induced by $\left( \omega_C, \Upsilon_C \right)$ on $\mathbb{R}_{>0} \times P$ is a metric cone:
    \[
        g_C = dt^2 + t^2 g^{se}.
    \]
\end{example}

 \section{\texorpdfstring{Circle-invariant $\rG_2$-structures}{Circle-invariant G₂-structures}}

\label{sec:s1Ansatz} 

In this section, we will describe circle-invariant $\rG_2$-structures in detail, before specialising to the conically singular case. 

Suppose we are given a vector field $\xi\in\sX(M^7)$ with compact orbits, i.e. $\xi$ generates an $S^1$-action on $M$, 
and a $\xi$-invariant $\rG_2$-structure $\varphi\in \Omega^3(M)$. There is a natural basic $\SU(3)$-structure $(\omega, \Upsilon)$ on the leaf space $N^6:=M/\xi$, which is compatible with $\varphi$ in the sense that the quotient map is a Riemannian submersion with respect to the induced metrics. This can be defined by $(\omega, \Im  \Upsilon)$: 
\begin{equation} 
\label{eq:CYformsbasic}
    \omega = \tfrac{1}{t} \, \xi \lrcorner \varphi 
    \qandq 
    \Im \Upsilon = -\tfrac{1}{t} \, \xi \lrcorner *\varphi,
\end{equation}
where $t:=|\xi|$ descends to a function on $N$, since $\xi$ preserves the metric $g_\varphi$ defined by $\varphi$. A result of \cite{Hitchin2001} ensures that the pair \eqref{eq:CYformsbasic} completely determines the $\SU(3)$-structure. 

Outside the set $S$ of fixed points of $\xi$, $M^7\setminus S \to N^6$ is a circle fibration
with a natural connection 1-form $\theta\in\Omega^1(M^7)$, defined by the formula 
\begin{equation}
\label{eq: theta 1-form}
    \theta(Y) 
    := t^{-2} g_\varphi(\xi, Y), \qforq Y \in \sX(M),
\end{equation}
and the Riemannian submersion then takes the form
\begin{align} \label{eq:metric} 
 g_\varphi= g + t^{2} \theta^2,
 \qwithq
 g := g_\varphi|_{\ker \theta}.
\end{align}
Moreover, the $\rG_2$-structure can then be written as 
\begin{equation} 
\label{eq:g2Ansatz4}
    \varphi = t \theta \wedge \omega + \Re{\Upsilon}
    \quad \qandq \quad 
    *\varphi = \tfrac{1}{2} \omega^2 - t \theta \wedge \Im \Upsilon.
\end{equation}
This is compatible with the orientation induced by the $\SU(3)$-structure, in the sense that the following volume forms match:  $\vol_\varphi = t \theta \wedge \vol_g = \tfrac{1}{6} t \theta \wedge \omega^3$. 

According to the splitting defined by the $\SU(3)$-structure $(\omega, \Upsilon)$ on $N$, the curvature 2-form of $\theta$ can be written as 
\begin{equation}
\label{eq: dtheta}
    d\theta 
    = X \lrcorner \Re{\Upsilon} + \lambda \omega + \sigma,
\end{equation}
for some vector field $X\in\sX(N)$, function $\lambda\in C^\infty(N)$, and $2$-form $\sigma\in \Omega^2_8(N)$.

\begin{remark}  \label{remark:conformal}
    Although less useful from the point of view of conical singularities, if one drops the requirement that the induced metric be a Riemannian submersion, we can obtain a whole family of $\SU(3)$-structures from $\varphi$ simply by conformally rescaling the $\SU(3)$-structure $(\omega, \Upsilon,g)$ by some function of $t$. The approach of Apostolov-Salamon \cite{Apostolov2004} chooses a conformal factor $h(t)$ such that requiring $\varphi$ to be closed also forces $\omega$ to be closed, hence $(N,\omega)$ is a symplectic manifold. 

    In \cite{Apostolov2004}, the associated almost-complex structure $J$ is assumed to be integrable, i.e. $(N,J,\omega)$ is Kähler: if $\varphi$ is torsion-free, this is equivalent to the requirement that the $\Omega^2_8$-component of $d \theta$ vanishes. In this case, the whole $\SU(3)$-structure must also be invariant under the vector field $J \nabla h$, and they consider $\SU(2)$-structures on the $4$-manifold arising from a symplectic reduction. 

    In \cite{FoscoloALC}, Foscolo-Haskins-Nordstrom forgo the integrability assumption in \cite{Apostolov2004} to find complete circle-invariant torsion-free $\rG_2$-structures. Instead, they take $(\omega,\Upsilon)$ to be a small perturbation of an asymptotically conical Calabi-Yau structure $(\omega_0,\Upsilon_0)$. Re-scaling $\theta$ by a small parameter $\epsilon>0$, they solve the re-scaled system for $(\omega_\epsilon,\Upsilon_\epsilon,h_\epsilon,\theta_\epsilon)$ on the non-compact symplectic manifold as formal power-series in $\epsilon$, and then show that this series converges by a fixed-point argument. 

\end{remark}

Recall \eqref{eq:CYtorsion}, where the torsion of the $\SU(3)$-structure is given in terms of differential forms on $N$. We will now write the torsion of $\varphi$ in terms of these basic data: i.e. the curvature $d \theta$, the torsion $(\upsilon_1, \hat{\upsilon}_1, \upsilon_6, \hat{\upsilon}_6, \upsilon_8, \hat{\upsilon}_8, \upsilon_{12})$ of $(\omega, \Upsilon)$, and the function $t$.

\begin{proposition} \label{prop:invariantg2torsion}
    Let $\xi$ be a vector field on $(M^7, \varphi)$ such that $\mathcal{L}_\xi \varphi =0$. Denote by $(\omega, \Upsilon)$ its natural basic $\SU(3)$-structure induced on $N^6=M/\xi$, with basic almost-complex structure $J$ and torsion $(\upsilon_1, \hat{\upsilon}_1, \upsilon_6, \hat{\upsilon}_6, \upsilon_8, \hat{\upsilon}_8, \upsilon_{12})$. Away from the fixed points of $\xi$, the torsion $(\tau_1, \tau_7, \tau_{14}, \tau_{27})$ of $\varphi$ is given by
\begin{align*}
   \tau_1 
    &=  \tfrac{6}{7} \left( t \lambda + 4 \hat\upsilon_1 \right), \\
    \tau_7 
    &= - t \upsilon_1 \theta + \tfrac{1}{6} Y_1, \\
    \tau_{14} 
    &=  \tfrac{2}{3}t \theta \wedge J Y_2  
   + \tfrac{1}{3} Y_2 \lrcorner \Re{\Upsilon} - \hat{\upsilon}_8, \\
    \tau_{27} 
    &= t \theta \wedge \left( 
     \tfrac{8}{7} ( t \lambda + \tfrac{1}{2} \hat{\upsilon}_1) \omega 
     - \tfrac{1}{2} J Y_3 \lrcorner \Re{\Upsilon} 
     - (\upsilon_8 + t \sigma) \right) 
    - \tfrac{1}{2} J Y_3 \wedge \omega 
     -  (\tfrac{1}{2}\hat{\upsilon}_1 + t \lambda)\tfrac{6}{7}\Re{\Upsilon} 
     - *\upsilon_{12},
\end{align*}
    where $\theta\in\Omega^1(M)$ is defined by \eqref{eq: theta 1-form},  $*\upsilon_{12}$ is the Hodge dual of $\upsilon_{12}$ with respect to the metric defined by $(\omega, \Upsilon)$, and we define the following vector fields on $M$:
\[
\begin{alignedat}{7}
 Y_1 &:=\;&-t J X 
     &\;+\;& \upsilon_6 
     & \;+\;& \hat{\upsilon}_6 
     &\;+\;& d\ln t, \\
Y_2 &:=\;&      +2t J X 
     &\; -\;& 2\upsilon_6 
     &\;+\;& \hat{\upsilon}_6 
     &\;+\;& d\ln t, \\
Y_3 &:=\;&\phantom{-}t J X 
     &\;+\;& \upsilon_6 
     &\;-\;& \hat{\upsilon}_6 
     &\;+\;& d\ln t.
\end{alignedat}
\]
 \end{proposition}

 \begin{proof}
 This computation uses \S \ref{sec:G2structures}, as well as the identities in  \S \ref{sec: SU(3)structures}, along with the observation that for any basic $k$-form $\alpha$, we have $*_7 ( \alpha \wedge t \theta) = (-1)^k *_6 \alpha$. For completeness, we include the first few lines of the computation:
 \begin{align*}
 d \varphi = & -t \theta \wedge \left( (d \ln{t} + \upsilon_6) \wedge \omega +3\upsilon_1 \Re{\Upsilon} + 3\hat{\upsilon}_1 \Im{\Upsilon}  + \upsilon_{12} \right) \\ &+ (t \lambda +2 \hat{\upsilon}_1) \omega^2 + (\hat{\upsilon}_6 - t JX) \wedge \Re{\Upsilon} + (\upsilon_8+ t\sigma) \wedge \omega, \\
 d *\varphi = & t \theta  \wedge \left( (d \ln{t} + \hat{\upsilon}_6) \wedge \Im{\Upsilon} - 2\upsilon_1 \omega^2 + \hat{\upsilon}_8 \wedge \omega \right)  \\
 &+  ( \upsilon_6 - t JX ) \wedge \omega^2.
 \end{align*}
 Taking the Hodge star of $d \varphi$, we have: 
 \begin{align*}
 * d \varphi = & -(J d\ln{t} + J \upsilon_6)\wedge \omega - 3 \upsilon_1 \Im{\Upsilon} + 3\hat{\upsilon}_1 \Re{\Upsilon} -  * \upsilon_{12} \\ &+ t \theta \wedge \left( 2(t \lambda+2 \hat{\upsilon}_1) \omega + (J \hat{\upsilon}_6 + tX) \lrcorner \Re{\Upsilon} - (\upsilon_8+ t\sigma) \right). 
 \end{align*}
Now, to compute $\tau_1, \tau_7$, we have:
 \begin{align*}
 d \varphi \wedge\varphi &= (t \lambda + 4 \hat\upsilon_1) t \theta \wedge \omega^3, \\
 *d\varphi \wedge \varphi &=  t\theta \wedge (J(d \ln{t} + \hat{\upsilon}_6 + \upsilon_6 - t J X)\wedge\omega^2 +2 \upsilon_1 \omega^3.
 \end{align*}
 Applying Lemma \ref{lem:g2torsion} gives $\tau_1, \tau_7$ as stated. Then for computing $\tau_{14}$, $\tau_{27}$:
 \begin{align*}
 4\tau_7 \wedge *\varphi&= \tfrac{1}{3}Y_1 \wedge \omega^2 + t\theta \wedge \left( \tfrac{2}{3} Y_1 \wedge \Im{\Upsilon} - 2 \upsilon_1 \omega^2 \right), \\  
 3\tau_7 \wedge \varphi &= \tfrac{1}{2} Y_1 \wedge \Re{\Upsilon} - t \theta \wedge \left( 3 \upsilon_1 \Re{\Upsilon} + \tfrac{1}{2} Y_1 \wedge \omega\right).  \qedhere
 \end{align*}
 \end{proof}

\begin{remark*} The decomposition of torsion given in Proposition \ref{prop:invariantg2torsion} is more general than is needed in this paper, but readers may find it of independent interest. 
\end{remark*}
\begin{corollary} 
\label{cor:invariantg2torsion}
The $\rG_2$-structure $\varphi$ defined by \eqref{eq:g2Ansatz4} is closed if, and only if, 
\begin{align*}
    d(t\omega) = 0 
    \qandq 
    d \Re{\Upsilon} = - t\, d\theta \wedge \omega,
\end{align*}
and co-closed if, and only if, 
\begin{align*}
    d(t\Im{\Upsilon}) = 0 
    \qandq 
    d\omega \wedge \omega = t\, d\theta \wedge \Im{\Upsilon}.
\end{align*}
\end{corollary}

\begin{proof}
By Proposition~\ref{prop:invariantg2torsion}, $\varphi$ is closed if, and only if
\begin{align*}
    \upsilon_1 = \hat{\upsilon}_1 = -t\lambda = \upsilon_{12} = 0, 
    \quad 
    \upsilon_6 = -t^{-1} dt, 
    \quad 
    \hat{\upsilon}_6 = t\, X \lrcorner \omega 
    \qandq 
    \upsilon_8 = -t \sigma,
\end{align*}
and co-closed if, and only if
\begin{align*}
    \upsilon_1 = \hat{\upsilon}_8 = 0, 
    \quad 
    \hat{\upsilon}_6 = -t^{-1} dt 
    \qandq 
    \upsilon_6 = t\, X \lrcorner \omega.
\end{align*}
\end{proof}


Suppose now that $\Sigma:=\Sigma_{t_0} \subset N$ is the level-set of a regular value $t_0>0$ in $N$, so that $N$ is diffeomorphic to $\R_{>0} \times \Sigma$ on an open interval about $t_0 \in \R_{>0}$, away from critical values. Restricted to this subset, $M \to \R_{>0} \times \Sigma$ is a circle fibration, with connection form $\theta$, and we have $|dt|>0$. 

If we denote $s:=|dt|$, then recalling \eqref{eq: CYdefine} in \S \ref{sec: SU(2)structures}, the  $\SU(3)$-structure \eqref{eq:g2Ansatz4} on $N$ can be written in terms of a one-parameter family of $\SU(2)$-structures $(\omega_1, \omega_2, \omega_3, \eta)_t$ on $\Sigma$ such that:
\begin{equation*}
    \omega = s^{-1} dt \wedge \eta + \omega_1
    \quad \qandq \quad
    \Upsilon = \left( s^{-1} dt + i \eta \right) \wedge \left( \omega_2 + i \omega_3 \right).
\end{equation*}
Using Corollary \ref{cor:invariantg2torsion}, we have the following lemma:
\begin{lemma} 
    Let $\dot{\theta} := \tfrac{\partial}{\partial t} \theta = \cL_{\partial_t} \theta$, and suppose that $t > 0$, $s := |dt| > 0$. Then $\varphi$ is closed if, and only if,  on $\Sigma$ 
    \begin{equation*}
        d \omega_1 = 0
        \quad \qandq \quad
        d (\eta \wedge \omega_3) = t \omega_1 \wedge d \theta,
    \end{equation*}
    subject to the evolution equations:
    \begin{equation*}
        \partial_t (t \omega_1) = t d (s^{-1} \eta)
        \quad \qandq \quad
        \partial_t (\eta \wedge \omega_3) - t  \omega_1 \wedge \dot{\theta}
        = t s^{-1} d \theta \wedge \eta - d (s^{-1} \omega_2).
    \end{equation*}
    Moreover, $\varphi$ is co-closed if, and only if,  on $\Sigma$
    \begin{equation*}
        d ( \eta \wedge \omega_2 ) = 0
        \quad \qandq \quad
        d \omega_1 \wedge \omega_1 = t d \theta \wedge \eta \wedge \omega_2,
    \end{equation*}
    subject to the evolution equations:
    \begin{equation*}
        \partial_t (t \eta \wedge \omega_2) = t d (s^{-1} \omega_3)
        \quad \qandq \quad
        \partial_t\omega_1 \wedge \omega_1 - t \dot{\theta} \wedge \eta \wedge \omega_2
        =  d ( s^{-1}\eta \wedge \omega_1) + t s^{-1} d\theta \wedge \omega_3.
    \end{equation*}
\end{lemma}
\begin{remark} By taking exterior derivatives, one can show that the conditions on $\omega_1, \omega_2, \omega_3, \eta$, and $\theta$ are preserved under evolution by their respective dynamic equations. In other words, it is sufficient that the conditions hold at some initial $t_0$ in the interval. 
\end{remark}
In particular, when $|dt|=1$ and the $\SU(3)$-structure is the cone over a fixed $\SU(2)$-structure on $\Sigma$:
\begin{align} 
\label{eq:CYcone2}
    \omega_C = t dt \wedge \eta + t^2 \omega_1
    \quad \qandq \quad
    \Upsilon_C = t^2 \left( \omega_2 + i \omega_3 \right) \wedge \left( dt + i t \eta \right).
\end{align}

\begin{corollary} 
\label{cor:g2torsioncone}
    Suppose $(\omega, \Upsilon)$ is given by \eqref{eq:CYcone2}. Then $\varphi$ is closed if, and only if, on $\Sigma$ 
    \begin{align} 
    \label{eq:staticclosed}
        d \eta = 3 \omega_1
        \quad \qandq \quad
        d (\eta \wedge \omega_3) = d \theta \wedge \omega_1,
    \end{align}
    with evolution equation:
    \begin{align} 
    \label{eq:evclosed}
        d \omega_2 + 3 \eta \wedge \omega_3 = t \dot{\theta} \wedge \omega_1 + d \theta \wedge \eta.
    \end{align}
    Moreover, $\varphi$ is co-closed if, and only if, on $\Sigma$ 
    \begin{align}  
    \label{eq:staticcoclosed}
        d \omega_3 = 4 \eta \wedge \omega_2
        \quad \qandq \quad
        d \theta \wedge \eta \wedge \omega_2 = 0,
    \end{align}
    with evolution equation:
    \begin{align}  
    \label{eq:evcoclosed}
        2 \omega_1^2 - d ( \eta \wedge \omega_1) 
        = d\theta \wedge \omega_3 + t \dot{\theta} \wedge \eta \wedge \omega_2.
    \end{align}
\end{corollary}
Recall from \S \ref{sec: SU(2)structures} that there is a Riemannian metric on the distribution $\cH := \ker \eta$ associated to the  $\SU(2)$-structure $(\eta, \omega_1, \omega_2, \omega_3)$. In the setting of the previous corollary, we have:
\begin{corollary} 
\label{lem:torsionfree1}
    Suppose $(\omega, \Upsilon)$ is given by \eqref{eq:CYcone2}. Then $\varphi$ is torsion-free if, and only if, the following holds: $(\eta, \omega_1, \omega_2, \omega_3)$ is a hypo-structure on $\Sigma$ satisfying
    \begin{align*}
        d \eta = 3 \omega_1 
        \quad \qandq \quad
        d \omega_3 = 4 \eta \wedge \omega_2,  
    \end{align*}
    and $(d \theta)_\Sigma = \omega_3 + \sigma$ is a $t$-dependent 2-form on $\cH$, with $\sigma$ anti-self dual. Moreover,
    \[
        \left. \dot{\theta} \right|_\cH = 0
        \qandq
        t \dot{\theta} \wedge \omega_1 
        = d \omega_2 + (3 \omega_3 - d \theta) \wedge \eta.
    \]  
\end{corollary}
In particular, in the situation of Apostolov-Salamon \cite{Apostolov2004}, i.e. $d \sigma = X \lrcorner \Re{\Upsilon}$ for some vector field $X$, we have strong restrictions on $(\eta, \omega_1, \omega_2, \omega_3)$:
\begin{lemma} 
    If $\varphi$ is torsion-free, then the torsion component $\upsilon_8$ of $(\omega_C, \Upsilon_C)$ vanishes if, and only if, $(\eta, \omega_1, \omega_2, \omega_3)$ also satisfies 
    $$d \omega_2 = - 4 \eta \wedge  \omega_3.$$ 
    In particular, the $\SU(2)$-structure $(\tfrac{4}{3} \eta, 2 \omega_1, 2 \omega_2, 2 \omega_3)$ is Sasaki-Einstein.   
\end{lemma}
\begin{proof} 
    We can re-write the first three of the torsion-free equations in Corollary~\ref{cor:invariantg2torsion} as
\begin{align}
\label{eq:tfcompact1}
    d (t \omega) = 0, \quad
    d (t \Im{\Upsilon}) = 0, 
    \quad
    d (t \Re{\Upsilon}) = t \upsilon_8 \wedge \omega,
\end{align}
with the components of $ d \theta = \lambda \omega + X \lrcorner \Re{\Upsilon} + \sigma$
satisfying
\begin{align}
\label{eq:tfcompact2}
    \sigma = -\tfrac{1}{t} \upsilon_8, \quad
    X = \tfrac{1}{t^2} J \nabla t, \quad
    \lambda = 0.
\end{align}
However $\upsilon_8 = 0$, if, and only if, $t \Re{\Upsilon}$ is closed, which is the case for $\Re{\Upsilon}_C$ if, and only if, $d \omega_2 = - 4 \eta \wedge  \omega_3$. 
\end{proof}
We will now give examples of solutions to the equations of Lemma \ref{cor:g2torsioncone} on the homogeneous space $$S^2 \times S^3 = \SU(2)^2/\triangle U(1).$$ 
Much of the set-up for $\SU(2)$-structures on $\SU(2)^2/\triangle \U(1)$ can be found in \cite{FoscolonK}, see also \cite{Stein2023}. 
We will use the basis $E_1, E_2, E_3$ for the Lie algebra $\mathfrak{su}(2)$ such that 
$$\left[ E_i, E_j\right] = 2E_k, \qforq (ijk)\sim (1 2 3),$$ 
and fix a basis for left-invariant vector fields on  $\SU(2)^2$: 
\begin{gather} \label{eq:basis}
\begin{aligned}
U^1 := (E_1, 0),&  &V^1 := (E_2, 0),&  &W^1 := (E_3, 0), \\
U^2 := (0, E_1),&  &V^2 := (0, E_2),&  &W^2 := (0, E_3),
\end{aligned}
\end{gather}
with respective dual 1-forms $u^1, v^1, w^1, u^2, v^2, w^2$. Let $U^\pm := U^1 \pm U^2$, with respective dual 1-forms $u^\pm$. Here, the vector field $U^+$ generates the diagonal subgroup $\triangle \U(1)$.

There is an invariant Sasaki-Einstein $\SU(2)$-structure on $\SU(2)^2/\triangle \U(1)$ given explicitly by
\begin{gather} \label{standardse} 
 \begin{aligned} 
\eta^{se} &:= \tfrac{4}{3} u^-, & \omega_1^{se} &:= - \tfrac{2}{3}(v^1 \wedge w^1 - v^2 \wedge w^2 ),  \\
\omega_2^{se} &:= \tfrac{2}{3}(v^1 \wedge v^2 + w^1 \wedge w^2 ), & \omega_3^{se} &:=  \tfrac{2}{3} (v^1 \wedge w^2 - w^1 \wedge v^2 ). \\   
 \end{aligned} 
 \end{gather}
 As is shown in \cite{FoscolonK}, the space of invariant 2-forms on $\SU(2)^2 / \triangle \U(1)$ is four-dimensional, and it is spanned by $\omega_0^{se}, \omega_1^{se},\omega_2^{se},\omega_3^{se}$, where we define: 
\begin{align}
\omega_0^{se} &:= \tfrac{2}{3}(v^1 \wedge w^1 + v^2 \wedge w^2 ).
\end{align} 
By using this basis of invariant 2-forms, and the invariant 1-form $\eta^{se}$, we have the following description of the space of 
$\SU(2)$-structures:
\begin{proposition}[\cite{FoscolonK}*{Proposition 2.11}] 
    For fixed orientation, we can identify the space of invariant non-degenerate $\SU(2)$-structures $\left(\eta, \omega_1, \omega_2,\omega_3 \right)$ on $\SU(2)^2 / \triangle \U(1)$ with $\R_{>0} \times \R_{>0} \times \SO_0(1,3)$, such that:
\begin{equation}
\label{hypoGEN}
    \eta = \lambda \eta^{se}
    \quad \qandq \quad
    \omega_i = \mu A_{ij} \omega_j^{se},
\end{equation}
    for some $\lambda>0, \mu>0$, $A \in \SO_0(1,3)$. 
\end{proposition} 

Moreover, if we let $\theta^{se}:= \tfrac{4}{3} u^+$ denote (up to scale) the canonical $SU(2)^2$-invariant abelian connection 1-form on $\SU(2)^2 \rightarrow \SU(2)^2 / \triangle \U(1)$, then by \cite{Wang1958} any $\SU(2)^2$-invariant abelian connection 1-form can be written, for some function $\alpha(t)$ and constant $k$, in the form
\begin{align} \label{conGEN}
    \theta 
    = \alpha \eta^{se}+ k \theta^{se}. 
\end{align}
The curvature of this connection on $\SU(2)^2/\triangle \U(1)$ is given by
\begin{align} \label{eq:curvature}
    d\theta 
    = 2 (\alpha \omega_1^{se} - k \omega_0^{se}). 
\end{align}
Using these facts alongside \cite{FoscolonK}*{Proposition 2.11}, we can compute the following: 
\begin{proposition} 
    Given an invariant $\SU(2)$-structure $(\lambda, \mu, A)$ as in \eqref{hypoGEN}, and a connection form $\theta$ as in \eqref{conGEN}, then the following hold:
\begin{enumerate}[(i)]
    \item $(\lambda, \mu, A, \theta)$ is a solution of \eqref{eq:staticclosed} and  \eqref{eq:evclosed} if, and only if, $\mu = \tfrac{2}{3} \lambda$, $\omega_1 = \mu \omega_1^{se}$, $A_{22} = \lambda A_{33} $, $A_{23} = -\lambda A_{32}$, $\alpha=0$, $k= -\tfrac{3}{2} \mu A_{30}$.
    
    \item $(\lambda, \mu, A, \theta)$ is a solution of \eqref{eq:staticcoclosed} and \eqref{eq:evcoclosed} if, and only if, $A_{20}=A_{21}=0$, $A_{22}= \tfrac{3}{4\lambda} A_{33}$,  $A_{23}= - \tfrac{3}{4\lambda} A_{32}$, $\mu -\lambda A_{11} = kA_{30} + \alpha A_{31}$.
\end{enumerate}
\end{proposition}
\begin{proof} First, we note that $\omega_i^{se} \wedge \omega_j^{se} = \delta_{ij} \left(\omega_1^{se}\right)^2$ where $\delta_{ij}$ denotes the Minkowski tensor of signature $(1,3)$. For part (i), the first equation of \eqref{eq:staticclosed} gives $\mu = \tfrac{2}{3} \lambda$, $\omega_1 = \mu \omega_1^{se}$. Since $\omega_0^{se}$ and $\omega_1^{se}$ are closed by \eqref{eq:curvature}, the second equation of \eqref{eq:staticclosed} implies that $\alpha = 0$. In particular $\dot\theta = 0$, so the evolution equation \eqref{eq:evclosed} gives the constraints $A_{22} = \lambda A_{33} $, $A_{23} = -\lambda A_{32}$.

For part (ii), the first equation of \eqref{eq:staticcoclosed} gives the conditions on $A_{2i}$ for $i=0,1,2,3$. Given that $A_{20}=A_{21}=0$, the second equation of \eqref{eq:staticcoclosed} imposes no further constraints. Since $\dot \theta = \dot \alpha \eta^{se}$ by \eqref{conGEN}, the evolution equation \eqref{eq:evcoclosed} imposes just $\mu -\lambda A_{11} = kA_{30} + \alpha A_{31}$. 
\end{proof}
Imposing both the closed and co-closed conditions, i.e. the full torsion-free condition, we recover the Bryant-Salamon cone over the nearly K\"{a}hler $S^3 \times S^3 = \SU(2)^3/\triangle \SU(2)$:
\begin{corollary} \label{cor:bryantsalamon}
    The $\SU(2)^2$-invariant solutions of the torsion-free equations on $\SU(2)^2/\triangle \U(1)$ constitute a one-parameter family with $\alpha = 0$, $k = \pm \tfrac{1}{2}$, $A_{30} = \mp \tfrac{\sqrt{3}}{3}$, $\mu = \tfrac{\sqrt{3}}{3}$, $\lambda = \tfrac{\sqrt{3}}{2}$. These are parametrised by a constant $\vartheta \in \left[0, 2 \pi\right)$ such that 
    \begin{equation*}
     \omega_2 = \tfrac{\sqrt{3}}{3}\left( \cos{\vartheta} \omega_2^{se} - \sin{\vartheta} \omega_3^{se} \right).   
    \end{equation*}
\end{corollary} 
\begin{remark} 
    We can account for the additional parameter by noticing that it stems from the Reeb vector field $U^-$ on $\SU(2)^2 / \triangle \U(1)$ acting on the total space of $S^3 \times S^3$. Moreover, we can fix the sign of the constant $k$ by the equivariant gauge transformation on $\SU(2)^2 \rightarrow \SU(2)^2 / \triangle \U(1)$ acting as multiplication by $-1 \in \U(1)$ on the fibres, which sends $\theta^{se} \mapsto - \theta^{se}$.  
\end{remark}
Up to such transformations, we see that if we fix the geometry on the base of the circle-fibration, i.e. the choice of an $\SU(2)$-structure on $\SU(2)^2 / \triangle \U(1)$, 
and we allow $\theta$ to vary, then for a co-closed $\rG_2$-structure the only possibility for the length of the circle fibre to be non-constant is that $A_{31}$ must vanish, so that $\alpha$ is unconstrained. We now show that one can recover a special one-parameter family of co-closed $\rG_2$-structures by imposing their evolution under the closed equations \eqref{eq:evclosed}. 
\begin{proposition} \label{thm:gamma} 
    If $(\lambda, \mu, A)$ is the nearly K\"{a}hler $\SU(2)$-structure on $S^3 \times S^3$, then there is a one-parameter family of solutions $\theta_\gamma$ to \eqref{eq:staticcoclosed}, \eqref{eq:evcoclosed},  \eqref{eq:evclosed} with $k = \pm \tfrac{1}{2}$ and $\alpha = \gamma t^{-3}$, for some constant $\gamma$. 
\end{proposition} 
\begin{proof} 
    With fixed $(\lambda, \mu, A)$ solving \eqref{eq:staticcoclosed} and \eqref{eq:evcoclosed}, such that $\omega_1 = \mu \omega_1^{se}$, the additional equations imposed by \eqref{eq:evclosed} are: 
\begin{equation*}
    k = -\tfrac{3}{2} \mu A_{30}
    \quad \qandq \quad
    \dot{\alpha} = -\tfrac{2\lambda}{t \mu} \alpha.
\end{equation*}
Solving the ODE with $\mu = \tfrac{\sqrt{3}}{3}$, $\lambda = \tfrac{\sqrt{3}}{2}$ gives the result. 
\end{proof}
This co-closed $\rG_2$-structure has an isolated, non-conical singularity at $t=0$, with circle fibres growing $\Theta(t^{-4})$ with respect to the Bryant-Salamon cone metric, and it is asymptotic to the cone as $t\rightarrow \infty$. To see this, we denote the Bryant-Salamon cone $\varphi_C, \psi_C$: the corresponding connection 1-form $\theta_C$ has $\dot\theta_C = 0$ by Corollary \ref{cor:bryantsalamon}, and so Corollary \ref{cor:g2torsioncone} implies that the corresponding $\SU(2)$-structure $(\eta, \omega_1,\omega_2,\omega_3)$ on $S^2 \times S^3$ satisfies $d(\omega_3 \wedge \eta)=0$. 

In particular, using \eqref{eq:CYcone2}, the one-parameter family of Proposition \ref{thm:gamma} produces a $\rG_2$-structure $(\varphi, \psi)$ such that 
\begin{equation*} \label{eq:singularg2structure}
    \varphi - \varphi_C 
    = \tfrac{2\sqrt{3}\gamma}{3} \, \omega_1 \wedge \eta 
    \quad \qandq \quad
    \psi - \psi_C 
    = d \left( \left( \tfrac{2\sqrt{3}\gamma}{3} \ln{t} \right) \omega_3 \wedge \eta \right),
\end{equation*}
where the exterior derivative on the right-hand side is on the total space of the cone. 
This is asymptotic to the Bryant--Salamon cone, since with respect to the induced metric of $(\varphi, \psi)$, we have
\begin{equation*}
    |\varphi - \varphi_C| = \Theta(t^{-3}) 
    \quad \qandq \quad 
    |\psi - \psi_C| = \Theta(t^{-4}),
    \qasq t \to \infty.
\end{equation*}
Meanwhile, using that $|\theta| = \tfrac{1}{t}$, the $\rG_2$-structure exhibits very different behaviour near the singular end:
\begin{equation*}
    |\varphi - \varphi_C| = \Theta(1) 
    \quad \qandq \quad 
    |\psi - \psi_C| = \Theta(t^{-1}),
    \qasq t\to0.
\end{equation*}
Note that the torsion of this $G_2$-structure behaves as $\Theta(t^{-4})$ as both the asymptotic and singular ends, and is of generic co-closed type, i.e. it has non-trivial torsion components $\tau_1$ and $\tau_{27}$. 

\section{Contact Calabi-Yau Ansatz}
\label{sec:cCYAnsatz}

Let us specialise the general set-up \eqref{eq:g2Ansatz4} of the previous section to a particular type of Sasakian $7$-manifolds, which were described and explored at length in the context of $\rG_2$-geometry e.g. in \cites{Calvo-Andrade2020,Lotay2022a,Lotay2023}. 
We let $M^7$ be equipped with a 6-dimensional distribution given as the kernel of some non-vanishing 1-form $\theta$, which admits a \emph{contact Calabi-Yau} transverse $\SU(3)$-structure $(\omega, \Upsilon)$ i.e. the transverse K\"{a}hler form is given by $d \theta = \omega$. Let $\xi$ denote the Reeb vector field on $M$ dual to $\theta$, i.e. $\theta ( \xi ) = 1$, $\mathcal{L}_\xi \theta = 0$, and denote the leaf space $N:= M / \langle \xi \rangle$. We will think of the transverse almost complex structure $J$ compatible with $(\omega, \Upsilon)$ on $\ker \theta$ as a $(1,1)$ tensor on $M$ by letting $J \xi = 0$. 

We will consider an Ansatz for a $\rG_2$-structure on $M$ given by varying the length of the Reeb orbits. Given moreover a smooth function $h:M \rightarrow \mathbb{R}_{\geq 0}$, we let:
\begin{align} \label{eq:g2Ansatz}
  \varphi = h \theta \wedge \omega + \Re \Upsilon,& &*\varphi = \tfrac{1}{2} \omega^2 - h \theta \wedge \Im  \Upsilon.
\end{align}
This $\rG_2$-structure induces a metric $g = g_{CY} + h^2 \theta^2$ on $M$, where $g_{CY}$ is the transverse Calabi-Yau metric induced by $(\omega, \Upsilon)$ on $\ker \theta$. When $h=h_0$ is a constant, this $\rG_2$-structure is co-closed, and the metric collapses to the Calabi-Yau in the limit $h_0 \rightarrow 0$.

Assuming the Reeb orbits are compact, i.e. $N$ admits the structure of a Riemannian orbifold, we wish to investigate whether we can achieve this collapse locally, at isolated points in $N$, while maintaining good control of the torsion. More specifically, since we have control over the torsion of \eqref{eq:g2Ansatz} when $h = h_0$ is constant, we seek local $L^\infty$ bounds.
Before computing the torsion of \eqref{eq:g2Ansatz}, we make a small remark: since $\xi$ is orthogonal to $\ker\theta$ with this metric, the linear operator $X \mapsto X^\perp := X - \tfrac{1}{h^2} g( X, \xi) \xi$ is the orthogonal projection of the tangent bundle onto $\ker \theta$. In particular, the transverse gradient of the function $h$ is given by $$\nabla^\perp h = \nabla h - \tfrac{\xi(h)}{h^2} \xi,$$
and its metric dual is
$$(dh)^\perp := (\nabla^\perp h)^\flat = dh - \xi(h) \theta.$$  

\begin{proposition} The torsion components $\tau_1, \tau_7, \tau_{14}, \tau_{27}$ of the $\rG_2$-structure \eqref{eq:g2Ansatz} are:
\begin{align*}
    &\tau_1 = \tfrac{6}{7}h, &&\tau_{14} = \tfrac{1}{3h} \left( \nabla h \lrcorner \Re \Upsilon + h \theta \wedge (J \nabla h)^\flat \right), \\
    &\tau_7 = \tfrac{1}{6h} \left(\nabla^\perp h\right)^\flat, 
    &&\tau_{27} = \tfrac{1}{h} (J \nabla h)^\flat \wedge \omega + \tfrac{8}{7} h^2 \theta \wedge \omega - \tfrac{1}{2} \left((J \nabla h) \lrcorner \Re \Upsilon  \right) \wedge \theta + \tfrac{6}{7} h  \Re \Upsilon.
\end{align*}
\end{proposition}
\begin{proof} 
    We compute
    \begin{align} 
\label{eq:torsioncompute}
    d \varphi = d h \wedge \theta \wedge \omega + h \omega^2 
    \quad \qandq \quad 
    d *\varphi = - d h \wedge \theta \wedge \Im \Upsilon.
\end{align}
    Using $*^2 = (-1)^{k(n-k)}$,  $X \lrcorner * \alpha = (-1)^k *( X^\flat \wedge \alpha )$ for any $\alpha \in \Omega^k$, and that $\theta^\sharp = \tfrac{1}{h^2} \xi$, we compute:
\begin{equation*}
    * d \varphi = - \tfrac{1}{2h} \nabla h \lrcorner \omega^2 + 2h\theta \wedge \omega
    \quad \qandq \quad 
    d^* \varphi = \tfrac{1}{h} \nabla h \lrcorner \Re \Upsilon.
\end{equation*}
    Finally, applying Lemma \ref{lem:g2torsion}  gives the result. 
\end{proof}

\begin{proposition}
\label{prop:g2ansatorsion} 
    The norms of the torsion components $\tau_1, \tau_7, \tau_{14}, \tau_{27}$ of the $\rG_2$-structure \eqref{eq:g2Ansatz}, with respect to the induced metric, are:
\begin{align*}
    &|\tau_1|^2 = \left(\tfrac{6}{7}\right)^2 h^2,
    & &|\tau_{14}|^2 = \tfrac{1}{3h^2} | \nabla^\perp h |^2, \\
    &|\tau_7|^2 = \tfrac{1}{36 h^2} | \nabla^\perp h |^2, 
    & &|\tau_{27}|^2 = \tfrac{6}{h^2} | \nabla^\perp h |^2 + \tfrac{48}{7} h^2. 
\end{align*}
\end{proposition}
\begin{proof} 
    We use that $|\theta| = \tfrac{1}{h}$, and the orthogonal splitting of forms on $\ker \theta$ given by \eqref{eq:CYforms}.  Note that terms in the formulae for torsion components are given with respect to this orthogonal splitting. 
\end{proof}

\begin{lemma} \label{lem:cCYisolated} 
    Suppose $M$ is equipped with a $G_2$-structure \eqref{eq:g2Ansatz}, with $\xi(h)=0$, then
    \begin{enumerate}[(i)]
    \item $|d^* \varphi|$ is bounded, if, and only if, $|d ( \ln h)|$ is bounded.
    \item If $h(p)=0$, $h$ is continuous at $p$, and is differentiable and positive on a punctured neighbourhood of $p$, then $|d ( \ln h)|$ cannot be bounded.
    \item If $h(x) = O(r^k)$ for $r = d(p,x)$ sufficiently small, $k\neq 0$, then $|d^* \varphi|=O(r^{-1})$.
    \end{enumerate}
\end{lemma}
\begin{proof} 
    Using the computation \eqref{eq:torsioncompute}, we have 
$$|d^* \varphi|^2 = 2|d (\ln h)^\perp |^2 = 2 | d \ln h |^2,
$$
which in the case $\xi(h)=0$ yields the first statement. For the second statement, suppose that $h(x)>0$ for $x \neq p$ in a neighbourhood of $p$. Then on any arc-length geodesic $\gamma(r)$ such that $\gamma(r) \rightarrow p$ as $r\rightarrow 0$, we have 
$$|d \ln h|_{\gamma(r)} \geq  \frac{1}{h} \frac{dh}{dr}>0
\qonq (0,\epsilon), 
$$ for $\epsilon>0$ sufficiently small. If we further assume $|d^* \varphi|$ is bounded, by part (i), this implies that 
$$0\leq \frac{dh}{dr} \leq C h
\qonq (0,\epsilon),$$
for some $C>0$. Applying Gr{\"o}nwall's inequality, this implies that $h (r) = 0$ on $(0,\epsilon)$, which would contradict $h(0)=0$ being  isolated.

The statement about $|d^* \varphi|$ for polynomial growth of $h$ follows analogously. 
\end{proof}

\begin{remark} 
    One might hope to control the norm of torsion arising from the derivative $dh$, by conformally rescaling the Ansatz \eqref{eq:g2Ansatz} along the base directions, cf. Remark \ref{remark:conformal}. However, since it is actually $d (\ln{h})$-terms that appear in the torsion, then the conclusion of the Lemma \ref{lem:cCYisolated} remains, with only minor changes to the proof.
\end{remark}



If we assume that the contact Calabi-Yau fibration is quasi-regular, i.e. the Reeb vector field $\xi$ has compact orbits, and moreover that $\xi(h)=0$, then there is an induced Riemannian submersion $\pi: M \rightarrow N$. Using the previous lemma, we show that one can never shrink fibres in this way to produce a smooth $\rG_2$-structure on $M^7$ with bounded torsion. In particular, if the resulting $\rG_2$-structure has bounded torsion, then not every geodesic ball $B_r \subset M$ can be \emph{volume-maximal}: i.e. $\vol (B_r) = \Theta(r^7)$, for sufficiently small $r$. However, every geodesic ball $B_r \subset M$ is volume-maximal for complete Riemannian manifolds by the Gauss Lemma, also for orbifolds and manifolds with isolated conical singularities, by definition \eqref{def:CSG2}. 

\begin{lemma} \label{lem:nonexistence}
    Denote by $d$ the metric distance on $N$. If $(M,\varphi)$ is volume-maximal in a neighbourhood of $\pi^{-1}(p)$, then $h(x) = \Theta(r^2)$, for all $x\in N$ such that $r= d(p, \pi(x))$ is sufficiently small. In particular, $(M,\varphi)$ cannot be a conically singular $\rG_2$-structure with $|d^* \varphi|$ bounded.
\end{lemma}

\begin{proof} Over some punctured geodesic ball $B_r \subset N$ of radius $r$, since we have a Riemannian submersion onto $B_r$, we can write the metric on $M$ as $$g = dr^2 + r^2 g_5 + h^2 \theta^2 + O(r^{2+\nu})$$ for the round metric $g_5$ on (a finite quotient of) $S^5$, and some $\nu>0$. Thus, in order for $\vol (B_r) = \Theta(r^7)$ to hold, we must have $h=O(r^2)$. Then $|d^* \varphi|$ cannot be bounded by the previous lemma. 
\end{proof}

\bibliographystyle{alpha}
\bibliography{Bibliografia-2024-04}

\end{document}